\documentclass[12pt,a4paper]{article}
\usepackage[utf8]{inputenc}
\usepackage{amsmath, amssymb, amsthm}
\usepackage{geometry}
\usepackage{hyperref}

\newtheorem{theorem}{Theorem}
\newtheorem{lemma}{Lemma}
\theoremstyle{definition}
\newtheorem{definition}{Definition}

\title{New Binomial Identities for Fibonacci, Lucas, and Generalized Fibonacci Sequences with Multiple Indices}
\author{Nick(s) Vorobtsov\thanks{Independent Researcher. Email: nvvorobtsov@mail.ru}}
\date{}

\begin{document}

\maketitle

\begin{abstract}
This paper presents new identities expressing the terms of Fibonacci, Lucas, and generalized Fibonacci sequences with multiple indices through powers of Lucas numbers and binomial coefficients. The obtained formulas rely on the application of symmetric polynomials (Waring's formulas) to the classical Binet's formula. Particular attention is given to the binomial expansion for the generalized Fibonacci sequence, which structurally combines two adjacent binomial coefficients from Pascal's triangle.
\end{abstract}

\section{Introduction}
The study of identities involving Fibonacci ($F_n$) and Lucas ($L_n$) numbers is a classic topic in number theory. Recent developments have shown that many such identities can be generalized using various combinatorial and algebraic methods. 

Similar identities for balancing and Lucas-balancing polynomials were obtained in [1] using generating functions, where Fibonacci and Lucas numbers appear as special cases. In this paper, we provide a direct derivation for Fibonacci and Lucas sequences by applying Waring's formulas for symmetric polynomials to Binet’s formula, offering a different combinatorial perspective and emphasizing the connection with Pascal’s triangle.

The Fibonacci and Lucas sequences are among the most well-known and widely studied objects in combinatorics and number theory. Their properties find applications in cryptography, algorithm theory, geometry, and the analysis of natural structures. One of the classical problems in the theory of linear recurrence sequences is finding relationships between elements with multiple indices (e.g., $F_{nm}$) and base elements ($F_n$, $L_n$). 

The motivation for this research was the search for universal polynomial representations for generalized Fibonacci numbers. Expressing $F_{nm}$, $L_{nm}$, and $G_{nm}$ as sums whose coefficients are binomial coefficients reveals new number-theoretic properties, explicitly connecting recurrence sequences with Chebyshev polynomials. The transition from recursive computations to explicit combinatorial sums allows for a deeper investigation of the algebraic structure of these sequences. In this paper, three new identities are rigorously proven. In particular, a formula for the generalized Fibonacci sequence is obtained, reduced to a symmetric form by applying an appropriate index shift.

\section{Basic Definitions and Lemmas}

\begin{definition}
The \textbf{Fibonacci sequence} $\{F_n\}$ is defined by the recurrence relation: $F_0 = 0, F_1 = 1$, and $F_n = F_{n-1} + F_{n-2}$ for $n \ge 2$.
\end{definition}

\begin{definition}
The \textbf{Lucas sequence} $\{L_n\}$ is defined by the relation: $L_0 = 2, L_1 = 1$, and $L_n = L_{n-1} + L_{n-2}$ for $n \ge 2$.
\end{definition}

\begin{definition}
The \textbf{generalized Fibonacci sequence} $\{G_n\}$ is defined by arbitrary initial conditions $G_0, G_1$, and the recurrence relation $G_n = G_{n-1} + G_{n-2}$ for $n \ge 2$.
\end{definition}

To prove the main theorems, we use the following classical results.

\begin{lemma}[Binet's Formula \cite{koshy2001}]
\label{lem:binet}
Let $\alpha = \frac{1 + \sqrt{5}}{2}$ and $\beta = \frac{1 - \sqrt{5}}{2}$ be the roots of the characteristic equation $x^2 - x - 1 = 0$. Then the explicit formulas are valid:
\begin{equation}
F_n = \frac{\alpha^n - \beta^n}{\sqrt{5}}, \quad L_n = \alpha^n + \beta^n.
\end{equation}
Since $\alpha \beta = -1$, the identity $\alpha^n \beta^n = (-1)^n$ holds.
\end{lemma}

\begin{lemma}[Waring's Formulas \cite{lidl1997}]
\label{lem:waring}
For any variables $x$ and $y$, their sum $S = x+y$ and product $P = xy$, the following identities express the sum of equal powers of $x$ and $y$ through elementary symmetric polynomials (establishing a connection with Chebyshev polynomials of the first and second kinds):
\begin{equation}
\label{eq:waring_2}
\frac{x^m - y^m}{x - y} = \sum_{i=0}^{\lfloor \frac{m-1}{2} \rfloor} (-1)^i \binom{m - 1 - i}{i} S^{m - 1 - 2i} P^i,
\end{equation}
\begin{equation}
\label{eq:waring_1}
x^m + y^m = \sum_{i=0}^{\lfloor m/2 \rfloor} (-1)^i \frac{m}{m - i} \binom{m - i}{i} S^{m - 2i} P^i.
\end{equation}
\end{lemma}

\begin{lemma}[d'Ocagne's Identity \cite{koshy2001}]
\label{lem:docagne}
For any integers $a$ and $b$, the following relation holds for Fibonacci numbers:
\begin{equation}
F_a F_{b-1} - F_{a-1} F_b = (-1)^{a+1} F_{b-a}.
\end{equation}
\end{lemma}

\section{Main Results}

\begin{theorem} \label{thm:fib}
For any integers $n > 0$ and $m > 0$, the following identity holds:
\begin{equation}
F_{nm} = F_n \sum_{i=0}^{\lfloor \frac{m-1}{2} \rfloor} \binom{m - 1 - i}{i} L_n^{m-1-2i} (-1)^{i(n+1)}.
\end{equation}
\end{theorem}
\begin{proof}
We use Binet's formula (Lemma \ref{lem:binet}), according to which $F_{nm} = \frac{\alpha^{nm} - \beta^{nm}}{\sqrt{5}}$, where $\alpha = \frac{1 + \sqrt{5}}{2}$ and $\beta = \frac{1 - \sqrt{5}}{2}$. We introduce notations for the $n$-th powers of the roots: $x = \alpha^n = \left(\frac{1 + \sqrt{5}}{2}\right)^n$ and $y = \beta^n = \left(\frac{1 - \sqrt{5}}{2}\right)^n$. In these notations, $F_{nm} = \frac{x^m - y^m}{\sqrt{5}}$. We factor out the algebraic term $x-y$:
$$ F_{nm} = \frac{x - y}{\sqrt{5}} \cdot \frac{x^m - y^m}{x - y}. $$
Since $\frac{x-y}{\sqrt{5}} = \frac{\alpha^n - \beta^n}{\sqrt{5}} = F_n$, the expression becomes $F_{nm} = F_n \cdot \frac{x^m - y^m}{x - y}$. To express the polynomial $\frac{x^m - y^m}{x - y}$ in terms of the Lucas number $L_n$, note that the sum of the variables $S = x+y = \alpha^n+\beta^n = L_n$, and their product $P = xy = (\alpha\beta)^n = (-1)^n$. Applying the first Waring's formula \eqref{eq:waring_2}:
$$ \frac{x^m - y^m}{x - y} = \sum_{i=0}^{\lfloor \frac{m-1}{2} \rfloor} (-1)^i \binom{m - 1 - i}{i} L_n^{m - 1 - 2i} ((-1)^n)^i. $$
Simplifying the exponent, we combine the signs: $(-1)^i \cdot (-1)^{ni} = (-1)^{i(n+1)}$. Substituting the transformed sum back into the expression for $F_{nm}$ yields the statement of the theorem.
\end{proof}

\begin{theorem} \label{thm:lucas}
For any integers $n > 0$ and $m > 0$, the following identity holds:
\begin{equation}
L_{nm} = \sum_{i=0}^{\lfloor m/2 \rfloor} \frac{m}{m - i} \binom{m - i}{i} L_n^{m - 2i} (-1)^{i(n+1)}.
\end{equation}
\end{theorem}
\begin{proof}
According to Binet's formula (Lemma \ref{lem:binet}), $L_{nm} = \alpha^{nm} + \beta^{nm}$, where $\alpha = \frac{1 + \sqrt{5}}{2}$ and $\beta = \frac{1 - \sqrt{5}}{2}$. Introducing the notations $x = \alpha^n = \left(\frac{1 + \sqrt{5}}{2}\right)^n$ and $y = \beta^n = \left(\frac{1 - \sqrt{5}}{2}\right)^n$, we obtain $L_{nm} = x^m + y^m$. The sum of the $m$-th powers of two variables can be represented as a polynomial of their sum $S = x+y = L_n$ and product $P = xy = (-1)^n$ using the second Waring's formula \eqref{eq:waring_1}. Substituting the values of $S$ and $P$, we get:
$$ L_{nm} = \sum_{i=0}^{\lfloor m/2 \rfloor} (-1)^i \frac{m}{m - i} \binom{m - i}{i} L_n^{m - 2i} ((-1)^n)^i. $$
Performing the sign transformation $(-1)^i \cdot (-1)^{ni} = (-1)^{i(n+1)}$ completes the proof of the theorem.
\end{proof}

\begin{theorem} \label{thm:general}
For the generalized Fibonacci sequence $\{G_n\}$, given $n > 0$ and $m > 0$, the following binomial identity holds:
\begin{equation}
\begin{split}
G_{nm} &= G_n \sum_{i=0}^{\lfloor \frac{m-1}{2} \rfloor} \binom{m - 1 - i}{i} L_n^{m - 2i - 1} (-1)^{i(n+1)} \\
       &+ G_0 \sum_{i=1}^{\lfloor \frac{m}{2} \rfloor} \binom{m - 1 - i}{i - 1} L_n^{m - 2i} (-1)^{i(n+1)}.
\end{split}
\end{equation}
\end{theorem}
\begin{proof}
An arbitrary term of the generalized sequence $G_k$ can be expressed through base Fibonacci numbers by the formula $G_k = G_1 F_k + G_0 F_{k-1}$, which is easily proven by induction. Applying this property to the index $nm$, we have:
$$ G_{nm} = G_1 F_{nm} + G_0 F_{nm-1}. $$
From the analogous relation for step $n$, it follows that $G_n = G_1 F_n + G_0 F_{n-1}$. We express $G_1$ from this: $G_1 = \frac{G_n - G_0 F_{n-1}}{F_n}$. Substituting $G_1$ into the expression for $G_{nm}$, we obtain:
$$ G_{nm} = \frac{G_n - G_0 F_{n-1}}{F_n} \cdot F_{nm} + G_0 F_{nm-1}. $$
Expanding the brackets and grouping the terms with $G_n$ and $G_0$, we find:
$$ G_{nm} = G_n \left( \frac{F_{nm}}{F_n} \right) + G_0 \left( \frac{F_n F_{nm-1} - F_{n-1} F_{nm}}{F_n} \right). $$
Consider the first term. The ratio $\frac{F_{nm}}{F_n}$ corresponds exactly to the sum proven in Theorem \ref{thm:fib}.
For the second term, consider the expression $F_n F_{nm-1} - F_{n-1} F_{nm}$. Applying d'Ocagne's identity (Lemma \ref{lem:docagne}) with $a=n$ and $b=nm$, the numerator transforms into $(-1)^{n+1} F_{nm-n} = (-1)^{n+1} F_{n(m-1)}$. Thus, the second term takes the form:
$$ G_0 \cdot (-1)^{n+1} \cdot \frac{F_{n(m-1)}}{F_n}. $$
Using Theorem \ref{thm:fib} for the expansion of $\frac{F_{n(m-1)}}{F_n}$ (replacing parameter $m$ with $m-1$), we have:
$$ \frac{F_{n(m-1)}}{F_n} = \sum_{j=0}^{\lfloor \frac{m-2}{2} \rfloor} \binom{m - 2 - j}{j} L_n^{m - 2 - 2j} (-1)^{j(n+1)}. $$
We bring the multiplier $(-1)^{n+1}$ inside the sum. Since $(-1)^{n+1} \cdot (-1)^{j(n+1)} = (-1)^{(j+1)(n+1)}$, we obtain:
$$ G_0 \sum_{j=0}^{\lfloor \frac{m-2}{2} \rfloor} \binom{m - 2 - j}{j} L_n^{m - 2 - 2j} (-1)^{(j+1)(n+1)}. $$
To bring both sums to a unified structural form, we perform a shift of the summation index by setting $i = j + 1$ (hence $j = i - 1$). The summation limits change to $i \in [1, \lfloor \frac{m}{2} \rfloor]$. The binomial coefficient takes the form $\binom{m-1-i}{i-1}$, the exponent of the Lucas number $L_n$ is $m - 2 - 2(i-1) = m - 2i$, and the sign multiplier becomes $(-1)^{i(n+1)}$. As a result, the second term exactly matches the second sum of the desired formula. The theorem is proven.
\end{proof}

\section{Conclusion}
In this paper, explicit binomial identities for computing multiple-index terms $F_{nm}, L_{nm}$, and $G_{nm}$ via $L_n$ are rigorously proven. It is shown that the application of Waring's formulas to the classical Binet's formula allows bypassing recursive computations and obtaining direct combinatorial representations. The transition to the binomial form in the generalized Fibonacci sequence is particularly notable: algebraic transformations demonstrate the reduction of two independent sums to the use of adjacent binomial coefficients. The obtained results can be applied to the algorithmic optimization of computations and are of independent interest in the theory of symmetric polynomials and combinatorics. Interactive materials, bilingual pages, and related publications are available on the author's website \cite{vorobtsov_portfolio}.

\end{document}